\documentclass[11pt,centertags]{amsart}
\usepackage{stmaryrd}
\usepackage{amscd}
\usepackage{easybmat}
\usepackage{mathrsfs}
\usepackage{amsfonts}
\usepackage{color}
\usepackage{pifont}
\usepackage{upgreek}
\usepackage{bm}
\usepackage{hyperref}
\usepackage{shorttoc}
\usepackage{amsmath,amstext,amsthm,a4,amssymb,amscd}
\usepackage[mathscr]{eucal}
\usepackage{mathrsfs}
\usepackage{epsf}
\numberwithin{equation}{section}

\def\p{\partial}

\def\b{\bar}

\def\mc{\mathcal}

\def\w{\wedge}

\def\t{\triangle}

\newtheorem{thm}{Theorem}[section]
\newtheorem{lemma}[thm]{Lemma}
\newtheorem{prop}[thm]{Proposition}

\theoremstyle{definition}
\newtheorem{rem}[thm]{Remark}

\theoremstyle{definition}
\newtheorem{defn}[thm]{Definition}
\newcommand{\comment}[1]{}
\begin{document}
\title{A Donaldson type functional on a holomorphic Finsler vector bundle}
\author[Huitao Feng]{Huitao Feng$^1$}
\author[Kefeng Liu]{Kefeng Liu$^2$}
\author[Xueyuan Wan]{Xueyuan Wan$^3$}
\address{Huitao Feng: Chern Institute of Mathematics \& LPMC,
Nankai University, Tianjin, China}

\email{fht@nankai.edu.cn}

\address{Kefeng Liu: School of Mathematics, Capital Normal University and Department of Mathematics, UCLA, Los Angeles, USA}

\email{liu@math.ucla.edu}

\address{Xueyuan Wan: Chern Institute of Mathematics \& LPMC,
Nankai University, Tianjin, China}

\email{xywan@mail.nankai.edu.cn}

\thanks{$^1$~Partially supported by NSFC (Grant No. 11221091, 11271062, 11571184)}

\thanks{$^3$~Partially supported by NSFC (Grant No. 11221091, 11571184)
and the Ph.D. Candidate Research Innovation Fund of Nankai University}

\begin{abstract}
  In this paper, we solve a problem of Kobayashi posed in \cite{Ko4} by introducing a Donaldson type functional on the space $F^+(E)$ of strongly pseudo-convex complex Finsler metrics on $E$ -- a holomorphic vector bundle over a closed K\"{a}hler manifold $M$. This Donaldson type functional is a generalization in the complex Finsler geometry setting of the original Donaldson functional and has Finsler-Einstein metrics on $E$ as its only critical points, at which this functional attains the absolute minimum.
\end{abstract}
\maketitle
\tableofcontents
\section*{Introduction}\label{s0}

The stability of holomorphic vector bundles is a concept in algebraic geometry, which was initiated by Mumford for constructing good moduli spaces for holomorphic vector bundles over algebraic curves, see \cite{Mum-F}. Later this notion was generalized to coherent sheaves by Takemoto \cite{Ta}, Bogomolov \cite{Bogo} and Gieseker \cite{Giese1}, \cite{Giese2}, etc. A differential geometric analogue to stability is the concept of a Hermitian-Einstein vector bundle over a closed K\"{a}hler manifold. The relationship between stability and the existence of Hermitian-Einstein metrics was realized by Yau soon after his celebrated work on Calabi conjecture. More precisely, the correspondence between stability and the existence of Hermitian-Einstein metrics asserts that a holomorphic vector bundle over a closed K\"{a}hler manifold is polystable if and only if it admits a Hermitian-Einstein metric. Note that early in 1965 Narasimhan and Seshadri \cite{Nara} considered the case of compact Riemannian surfaces, and in 1980s Kobayashi \cite{Ko2} and L\"{u}bke \cite{Lubke} proved that a holomorphic bundle admitting a Hermitian-Einstein metric must be polystable. The truly hard part of this correspondence is the reversed problem and this was solved by Donaldson \cite{Don1}, \cite{Don2} for algebraic manifolds, and by Uhlenbeck and Yau \cite{Yau} for general K\"{a}hler manifolds. This well-known Donaldson-Uhlenbeck-Yau theorem states that the stability of a holomorphic vector bundle implies the existence of Hermitian-Einstein metric.

For the semi-stable case, Mabuchi pointed out explicitly the notion of approximate Hermitian-Einstein metrics, which was already implicit in Donaldson's works. Kobayashi proved that the existence of the approximate Hermitian-Einstein structure implies semi-stable for a holomorphic vector bundle. Moreover, over the projective algebraic manifolds, he solved the reversed part, that is, semi-stability implies the existence of the approximate Hermitian-Einstein structure. He conjectured that this result should be true for general closed K\"{a}hler manifolds. In \cite{Jacob}, A. Jacob gave a proof of this conjecture.

In \cite{Ko4} Kobayashi extended the concept of the Hermitian-Einstein metric to the setting of complex Finsler geometry. He introduced the definition of a Finsler-Einstein metric on a holomorphic vector bundle in terms of the Kobayashi curvature, as named in our paper \cite{FLW}, of a Finsler metric. Furthermore, Kobayashi raised in \cite{Ko4} the following open problem:

{\bf Problem:\,}\emph{What are the algebraic geometric consequences of the Finsler-Einstein condition? The first question in
this regard is whether every Einstein-Finsler vector bundle is semi-stable or not?}

Note that T. Aikou \cite{Aikou1} also introduced a definition for Finsler-Einstein vector bundles. When the Finsler metric comes from a Hermitian metric, his definition coincides with Kobayashi's. In general case, Kobayashi's notion is weaker than Aikou's. In \cite{FLW}, we proved, under an extra assumption, that a Finsler-Einstein vector bundle is semistable by using the Chern forms constructed in \cite{FLW}.

In this paper, we always assume that $M$ is a closed $n$-dimensional K\"{a}hler manifold with a K\"{a}hler metric $\omega$, and $E$ is a rank $r$ holomorphic vector bundle over $(M,\omega)$, and we use $F^+(E)$ to denote the space of all strongly pseudo-convex Finsler metrics on $E$.

Motivated by Donaldson's works on Hermitian-Einstein metrics (cf. \cite{Don1}, \cite{Don2}), we introduce a Donaldson type functional $\mathcal{L}$ on $F^+(E)$, which is a natural generalization of the famous Donaldson functional in the Hermitian-Einstein case. In this paper we solve the above Kobayashi problem by proving the following theorem,
\begin{thm}\label{0.1} Let $E$ be a holomorphic vector bundle over a closed K\"{a}hler manifold $(M,\omega)$. If $E$ admits a Finsler-Einstein
metric, then $E$ must admit an $\omega$-approximate Hermitian-Einstein structure. As a consequence, $E$ is $\omega$-semistable.
\end{thm}

A natural question emerged is whether a semi-stable holomorphic vector bundle over a closed K\"{a}hler manifold admits a Finsler-Einstein metric. A possible approach to this problem is to have a more comprehensive understanding of Donaldson's method in the Hermitian-Einstein case.

This paper is arranged as follows. In Section 1, we review briefly some definitions, notations in complex Finsler geometry used in this paper. In Section 2, by introducing a metric on $F^+(E)$, we introduce an energy functional $\mathcal{E}(G)$ on the path space $\Omega(G_0,G_1)$ of all piecewise smooth curves $G_t$, $0\leq t\leq 1$, in $F^+(E)$ with the two fixed endpoint $G_0,G_1\in F^+(E)$. Then we compute the first variation of this energy functional, from which a geodesic equation on the space $F^+(E)$ is derived and a family case of the geodesic approximate lemma in Chen's paper (cf. \cite{Chen}, Lemma 7) is given. In Section 3, we introduce a Donaldson type functional $\mathcal{L}$ on $F^+(E)$ and compute the first and second variations of the functional $\mathcal{L}$. By these variation formulas, we can prove that the Finsler-Einstein metrics on $E$ coincide with the critical points of this functional, and at which the Donaldson type functional $\mathcal{L}$ attains the absolute minimum, from which our main result Theorem 0.1 follows.
$$ \ $$
\noindent{\bf Acknowledgement:} The authors would like to thank Professor Jixiang Fu in Fudan University and Professor Xi Zhang in University of Science and Technology of China for very helpful discussions on the complex Monge-Amp\`{e}re equations. The first author would like to thank Professor Weiping Zhang in Chern Institute of Mathematics, Nankai University for the encouragement on his study of Finsler geometry.

\section{A brief review of complex Finsler vector bundles}\label{brief}

In this section we give a brief review of holomorphic Finsler vector bundles. One can consult with the references \cite{Aikou2}, \cite{Cao-Wong}, \cite{FLW}, \cite{Ko1}, \cite{Ko3} for more detail. In this paper, we will use the notations as in \cite{FLW}.

Let $\pi:E\to M$ be a rank $r$ holomorphic vector bundle over an n-dimensional complex manifold $M$. Let $G$ be a strongly pseudo-convex complex Finsler metric on $E$, that is, $G$ is a non-negative continuous function on $E$ verifying the following conditions
\begin{description}
  \item [F1)] $G$ is smooth on $E^o=E\setminus O$, where $O$ denotes the zero section of $E$;
  \item[F2)] $G(z,v)\geq 0$ for all $(z,v)\in E$ with $z\in M$ and $v\in\pi^{-1}(z)$, and $G(z,v)=0$ if and only if $v=0$;
  \item[F3)] $G(z,\lambda v)=|\lambda|^2G(z,v)$ for all $\lambda\in\mathbb{C}$;
  \item[F4)] the Levi form ${\sqrt{-1}}\partial\bar\partial G$ on $E^o$ is positive-definite along fibres $E_z=\pi^{-1}(z)$ for $z\in M$.
\end{description}

Let $P(E)$ denote the projectivized space of $E$, which is a holomorphic fibration $\pi:P(E)\to M$ with the fibres $P(E_z)={(E_z\setminus \{0\})}/{\mathbb{C}^*}$, $z\in M$. By the Kobayashi correspondence as named in \cite{FLW}, every strongly pseudo-convex complex Finsler metric $G$
induces a Hermitian metric on the tautological line bundle $\mathcal{O}_{P(E)}(-1)\to P(E)$, which we will still denote by $G$. So $G^{-1}$ gives a dual metric on the dual line bundle $\mathcal{O}_{P(E)}(1)$ and the first Chern form of $(\mathcal{O}_{P(E)}(1),G^{-1})$ is given by
\begin{align}\label{1.0}
\Xi(G)={{\sqrt{-1}}\over{2\pi}}\partial\bar\partial\log G.
\end{align}
Moreover, one has (cf. \cite{FLW})
\begin{align}\label{1.00}
\Xi(G)=-{1\over{2\pi}}\Psi(G)+\omega_{FS}(G),\quad \int_{P(E)/M}\Xi(G)^{r-1}=1,
\end{align}
where $\Psi(G)$ is the Kobayashi curvature of $G$ and $\omega_{FS}(G)$ is a well-defined vertical real $(1,1)$-form on $P(E)$; when restricted to each fibre $P(E_z)$ of $P(E)$, the real $(1,1)$-form $2\pi\omega_{FS}(G)$ gives a K\"{a}hler metric on $P(E_z)$, for any $z\in M$. So the vertical holomorphic tangent bundle $\tilde{\mathcal{V}}$ of $P(E)$ is a Hermitian vector bundle. In this paper, we often omit  $G$ from the notation $\Xi(G)$, $\Psi(G)$ and $\omega_{FS}(G)$. In local homogenous coordinate systems $(z^\alpha,v^i)$ on $P(E)$, one has
\begin{align}\label{1.1}
\Psi={\sqrt{-1}}K_{i\bar j \alpha\bar\beta}{{v^i\bar v^j}\over G}dz^\alpha\wedge d\bar z^\beta,\ \omega_{FS}={{\sqrt{-1}}\over{2\pi}}{{\partial^2\log G}\over{\partial v^i\partial\bar v^j}}\delta v^i\wedge\delta\bar v^j,
\end{align}
where
\begin{align}\label{1.2}
\delta v^i=dv^i+\Gamma^i_{j\alpha}v^jdz^\alpha,\ \Gamma^i_{j\alpha}={{\partial G_{j\bar k}}\over{\partial z^\alpha}}G^{\bar k i},\ K_{i\bar j \alpha\bar\beta}=-G_{i\bar j \alpha\bar\beta}+G^{k\bar l}G_{i\bar l\alpha}G_{k\bar j\bar\beta}.
\end{align}

For any local homogenous coordinate systems $(z^\alpha,v^i)$ on $P(E)$ and any smooth real function $f$ on $P(E)$, set
\begin{align}\label{1.2222}
\p^{\mathcal{V}}f=\sum^r_{i=1}{{\partial f}\over{\partial v^i}}\delta v^i,\
\p^{\mathcal{H}}f=\sum^n_{\alpha=1}{{\delta f}\over{\delta z^\alpha}}dz^\alpha,\quad\forall\ f\in C^\infty(P(E)),
\end{align}
where
\begin{align}\label{1.2223}
{{\delta}\over{\delta z^\alpha}}={{\partial}\over{\partial z^\alpha}}-\Gamma^i_{j\alpha}v^j{{\partial}\over{\partial v^i}}.
\end{align}
The one-forms $\p^{\mathcal{V}}f$ and $\p^{\mathcal{H}}f$ do not depend on the local homogenous coordinate $(z^\alpha,v^i)$. Hence $\p^{\mathcal{V}}$ and $\p^{\mathcal{H}}$ are two well-defined operators acting on $C^\infty(P(E))$, the space of smooth real functions on $P(E)$, and are called the vertical and the horizontal derivatives, respectively. We will use $\|\p^\mathcal{V}f\|$ to  denote the (vertical) norm of $\p^\mathcal{V}f$.
Furthermore, when $M$ is a K\"{a}hler manifold with a K\"{a}hler form $\omega={\sqrt{-1}}g_{\alpha\bar\beta}dz^\alpha\wedge d\bar z^\beta$, the pull-back of $\omega$ is a horizontal real $(1,1)$-form and equips the horizontal tangent bundle $\tilde{\mathcal{H}}$ of $P(E)$ a Hermitian metric.
By this one can define the (horizontal) norm $\|\p^{\mathcal{H}}f\|$ of $\p^{\mathcal{H}}f$. In a local homogenous coordinate system $(z^\alpha,v^i)$
of $P(E)$, one has
\begin{align}\label{1.22233}
\|\p^\mathcal{V}f\|^2=GG^{i\b j}{\frac{\p f}{\p v^i}}{\frac{\p f}{\p\b v^j}},\quad \|\p^\mathcal{H}f\|^2=g^{\alpha\b{\beta}}\frac{\delta f}{\delta z^{\alpha}}\frac{\delta f}{\delta \b{z}^{\beta}}.
\end{align}
We have the following lemma for these two norms.
\begin{lemma}\label{l22} For any $f\in C^\infty(P(E))$, one has
\begin{align}\label{1.aaa}
\frac{(r-1)\sqrt{-1}}{2\pi}(\p^\mathcal{V}f\wedge\b\p^\mathcal{V} f)\w\omega_{FS}^{r-2}=\|\p^{\mathcal{V}}f\|^2\omega_{FS}^{r-1},
\end{align}
\begin{align}\label{1.bbb}
n{\sqrt{-1}}(\p^\mathcal{H}f\wedge\b\p^\mathcal{H}f)\w\omega^{n-1}=\|\p^{\mathcal{H}}f\|^2\omega^n.
\end{align}
\end{lemma}
\begin{proof} We only prove (\ref{1.aaa}), while (\ref{1.bbb}) is similar. Since the equality (\ref{1.aaa}) is pointwise, we will verify it at any point $(z,[v])\in P(E)$. Let $V_1,\cdots,V_{r-1}$ be a unitary basis of $\tilde{\mathcal{V}}|_{(z,[v])}$ with respect to $2\pi\omega_{FS}$, and let $\psi^1,\cdots,\psi^{r-1}$ be the dual basis. Then at the point $(z,[v])\in P(E)$ we have
\begin{align}\label{1.ccc}
\omega_{FS}=\frac{{\sqrt{-1}}}{2\pi}(\psi^1\w\b\psi^1+\cdots+\psi^{r-1}\w\b\psi^{r-1}).
\end{align}
For any $f\in C^\infty(P(E))$, we have
\begin{align}\label{1.cccc}
\p^\mathcal{V}f=\sum_{i=1}^{r-1}(V_if)\psi^i,\quad \b\p^\mathcal{V}f=\sum_{i=1}^{r-1}(\b V_if)\b\psi^i.
\end{align}
Clearly, one has
\begin{align}\label{1.ccccc}
\|\p^\mathcal{V}f\|^2=\sum_{i=1}^{r-1}|V_if|^2,\quad \p^\mathcal{V}f\wedge\b\p^\mathcal{V}f=\sum^{r-1}_{i=1}|V_if|^2\psi^i\w\b\psi^i+\sum_{i\neq j}(V_if)(\b V_jf)\psi^i\w\b\psi^j.
\end{align}
Now (\ref{1.aaa}) follows from the following direct computations:
\begin{align*}
&\frac{(r-1){\sqrt{-1}}}{2\pi}(\p^\mathcal{V}f\wedge\b\p^\mathcal{V} f)\w\omega_{FS}^{r-2}\\
&=(r-1)(\frac{{\sqrt{-1}}}{2\pi})^{r-1}\left(\sum^{r-1}_{i=1}|V_if|^2\psi^i\w\b\psi^i+\sum_{i\neq j}(V_if)(\b V_jf)\psi^i\w\b\psi^j\right)(\sum^{r-1}_{i=1}\psi^i\w\b\psi^i)^{r-2}\\
&=(r-1)(\frac{{\sqrt{-1}}}{2\pi})^{r-1}(\sum^{r-1}_{i=1}|V_if|^2\psi^i\w\b\psi^i)(\sum^{r-1}_{i=1}\psi^i\w\b\psi^i)^{r-2}\\
&=(r-1)!(\frac{{\sqrt{-1}}}{2\pi})^{r-1}\sum^{r-1}_{i=1}|V_if|^2\psi^1\w\b\psi^1\w\cdots\w\psi^{r-1}\w\b\psi^{r-1}\\
&=\sum^{r-1}_{i=1}|V_if|^2\omega_{FS}^{r-1}=\|\p^{\mathcal{V}}f\|^2\omega_{FS}^{r-1}.
\end{align*}
\end{proof}
In \cite{Ko4} Kobayashi introduced a notion of a Finsler-Einstein metric on $E\to (M,\omega)$ by using the Kobayashi curvature $\Psi$.
\begin{defn}\label{1.886}(Kobayashi \cite{Ko4}) Let $(M,\omega)$ be a closed K\"{a}hler manifold with the K\"{a}hler form $\omega={\sqrt{-1}}g_{\alpha\bar\beta}dz^\alpha\wedge d\bar z^\beta$. Let $E$ be a holomorphic vector bundle over $M$. A strongly pseudo-convex Finsler metric $G$ on $E$ is said to be Finsler-Einstein if there exists some constant $\lambda$ such that
\begin{align}\label{d1.10}
tr_{\omega}\Psi:=g^{\alpha\bar\beta}K_{i\bar j \alpha\bar\beta}{{v^i\bar v^j}\over G}=\lambda.
  \end{align}
\end{defn}
As in the Hermitian-Einstein case, the constant $\lambda$ in (\ref{d1.10}), if exists, must be
\begin{align}\label{2.000}
\lambda=\frac{2\pi n}{r}{{\int_{M}c_1(E)\w\omega^{n-1}}\over{\int_{M}\omega^{n}}}.
\end{align}
Clearly $\lambda$ only depends on the first Chern class $c_1(E)$ and the K\"{a}hler class $[\omega]$ on $M$.
\begin{rem}
  If the Finsler metric $G$ comes from a Hermitian metric $h$, then $K_{i\b{j}\alpha\b{\beta}}$ is independent of the fiber coordinates. Differentiating both sides of the formula $g^{\alpha\b{\beta}}K_{i\b{j}\alpha\b{\beta}}v^i\b{v}^j=\lambda G$ with respect to $v^i$, we get
  $g^{\alpha\b{\beta}}K_{i\b{j}\alpha\b{\beta}}=\lambda G_{\alpha\b{\beta}}$ or $g^{\alpha\b{\beta}}K^i_{\b j\alpha\b{\beta}}=\lambda \delta^i_j$. So $G$ is a Hermitian-Einstein metric on $E$ (cf. \cite{Ko3}). Here the tensor $K_{i\bar j}:=g^{\alpha\b{\beta}}K_{i\b j\alpha\b{\beta}}$ or the endomorphism $K(h)=(K^i_{\b j})=(g^{\alpha\b{\beta}}K^i_{\b j\alpha\b{\beta}})$ is usually called the mean curvature of the Hermitian metric $h$.
\end{rem}
Finally, we recall the notion of the $\omega$-approximate Hermitian-Einstein metric (cf. \cite{Ko3}).
\begin{defn}\label{App def} Let $E$ be a holomorphic Hermitian vector bundle over a closed K\"{a}hler manifold $(M,\omega)$.
The bundle $E$ is said to admit an $\omega$-approximate Hermitian-Einstein structure if for every positive $\epsilon$, there is a Hermitian structure $h_\epsilon$ such that the mean curvature $K(h_\epsilon)$ verifies that
\begin{align}\label{1.999}
\max_M|K(h_\epsilon)-\lambda I_E|<\epsilon,
\end{align}
where $\lambda$ is given by (\ref{2.000}), and $|K-\lambda I_E|$ is defined as
\begin{align}\label{1.9991}
|K-\lambda I_E|^2=tr((K-\lambda I_E)\circ(K-\lambda I_E)).
\end{align}
Such a family of metrics $h_\epsilon$ is called an $\omega$-Hermitian-Einstein structure on $E$.
\end{defn}

\section{An Energy functional $\mathcal{E}(G)$ on the path spaces $\Omega(G_0,G_1)$}\label{space}

In this section, we introduce an energy functional $\mathcal{E}(G)$ on the path space $\Omega(G_0,G_1)$ of all piecewise smooth curves $G_t$, $0\leq t\leq 1$, in $F^+(E)$ with the two fixed endpoint $G_0,G_1\in F^+(E)$. We will compute the first variation of $\mathcal{E}(G)$, from which a geodesic equation on the space $F^+(E)$ is derived and a family version of the geodesic approximation lemma in \cite{Chen} is given.

Let $F(E)$ be the space consisting of real continuous functions $H:E\to \mathbb{R}$ verifying the conditions ${\bf F1), F3)}$ given in Section 1. Clearly, $F(E)$ can be identified with $C^\infty(P(E))$ and so is a real topological vector space of infinite dimension. By the strongly pseudo-convex condition {\bf F4)}, the space $F^+(E)$ is a path-connected open subset in $F(E)$. Note that for any $G\in F^+(E)$, $F(E)$ can be  identified naturally with the tangent space $T_GF^+(E)$ of $F^+(E)$ at $G$, and when restricted to the fibre $P(E_z)$ of $P(E)$ for each $z\in M$, the $(r-1,r-1)$-form $\Xi^{r-1}$ gives a normalized volume form on $P(E_z)$ by (\ref{1.00}). Recall that $\omega$ is the K\"{a}helr form on the closed K\"{a}hler manifold $M$. So we can define a Riemannian metric on $F^+(E)$ as following: for any $G\in F^+(E)$ and any $H_1,H_2\in F(E)\equiv T_GF^+(E)$, we define $(H_1, H_2):=(H_1, H_2)_G$ by
\begin{align}\label{1.3}
   (H_1,H_2)_G=\int_M\langle H_1, H_2\rangle_G\frac{\omega^{n}}{n!},
 \end{align}
where
 \begin{align}\label{1.4}
 \langle H_1, H_2\rangle_G=(r+1)\int_{P(E)/M}{\frac{H_1}{G}}{\frac{H_2}{G}}\Xi^{r-1}
   -r\int_{P(E)/M}\frac{H_1}{G}\Xi^{r-1}\int_{P(E)/M}\frac{H_2}{G}\Xi^{r-1}.
\end{align}
Note that the positivity of $(\cdot,\cdot)$ comes from
   \begin{align*}
   \begin{split}
     \langle H_1, H_1\rangle_G&=\int_{P(E)/M}(\frac{H_1}{G})^2\Xi^{r-1}+r\left(\int_{P(E)/M}(\frac{H_1}{G})^2\Xi^{r-1}-
     (\int_{P(E)/M}\frac{H_1}{G}\Xi^{r-1})^2\right)\\
     &\geq \int_{P(E)/M}(\frac{H_1}{G})^2\Xi^{r-1},
     \end{split}
   \end{align*}
where the last inequality is by the Cauchy-Schwarz inequality and $\int_{P(E)/M}\Xi^{r-1}=1$.

In the rest of this paper, we will often omit the subscript $G$ from the inner product defined above.
\begin{rem}\label{r111} As in \cite{Mab}, the space $F^+(E)$ can be viewed as an infinite dimensional Riemannian manifold with the Riemannian metric $(\cdot,\cdot)_G$. For any smooth path $G_t,\ 0\leq t\leq 1$, in $F^+(E)$, we can define the covariant derivative along the path $G_t$ by
\begin{align}\label{3.99999}
  \frac{D}{\p t}:=\frac{\p}{\p t}-\frac{1}{2} G^{i\b{j}}\left(\frac{\p^2 G}{\p \b{v}^j\p t}\frac{\p}{\p v^i}+\frac{\p^2 G}{\p v^i\p t}\frac{\p}{\p \b{v}^j}\right).
\end{align}
We can prove this connection is metric-preserving and torsion-free, and moreover, the curvature of this connection is given by
 \begin{align}\label{3.99998}
    R(H_1,H_2)_G(H_3)=-\frac{1}{4}[[H_1,H_2]_G,H_3]_G
  \end{align}
for all $H_1, H_2, H_3\in F(E)$ and $G\in F^+(E)$, where
\begin{align}\label{3.99997}
  [H_1,H_2]_G:=G^{i\b{j}}\left(\frac{\p H_1}{\p v^i}\frac{\p H_2}{\p\b{v}^j}-\frac{\p H_2}{\p v^i}\frac{\p H_1}{\p\b{v}^j}\right).
  \end{align}
We will leave the further study on this issue to another paper.
\end{rem}
Using the $L^2$ metric $(\cdot,\cdot)$ defined above, we introduce the following energy functional $\mathcal{E}(G)$ on $\Omega(G_0,G_1)$ on $F^+(E)$ for any two fixed point $G_0,G_1\in F^+(E)$:
\begin{align}\label{1.6}
\mc{E}(G)=\frac{1}{2}\int_0^1|\p_t G|^2 dt,
\end{align}
where $|\p_t G|$ is the $L^2$-norm of $\p_t G:={\p G}/{\p t}\in F(E)$ defined by (\ref{1.3}).

A critical point of $\mathcal{E}(G)$ is said to be a geodesic in $F^+(E)$ joint $G_0$ and $G_1$.

Next, we compute the first variation of the energy functional $\mathcal{E}(G)$. Giving any infinitimal variation $V=V(t)$, which is a smooth curve in
$F(E)$ with $V(0)=V(1)=0$, we have the following first variation formula of $\mathcal{E}(G)$
\begin{thm}\label{t1.1}
\begin{align}\label{1.7}
{{\frac{d\mathcal{E}(G+sV)}{ds}}}|_{s=0}=-\int_0^1\left(G(\p_tv_t-\left\|\p^{\mathcal{V}}v_t\right\|^2),V\right)dt,
\end{align}
where $\partial_t$ denotes ${\partial\over{\partial t}}$ and $v_t=\partial_t\log G$.
\end{thm}
\begin{proof} Note that
\begin{align}\label{1.8}
\begin{split}
\mathcal{E}(G+sV)&=\frac{1}{2}\int_0^1\int_M \left[(r+1)\int_{P(E)/M}(\p_t\log(G+sV))^2\left(\frac{\sqrt{-1}}{2\pi}\p\b{\p}\log(G+sV)\right)^{r-1}\right.\\
   &\left.-r\left(\int_{P(E)/M}(\p_t\log(G+sV))\left(\frac{\sqrt{-1}}{2\pi}\p\b{\p}\log(G+sV)\right)^{r-1}\right)^2\right]\frac{\omega^n}{n!}dt.
\end{split}
\end{align}
By a direct computation, one has
\begin{align}\label{1.9}
   \begin{split}
     \frac{d\mathcal{E}(G+sV)}{ds}|_{s=0}&=\int_0^1\int_M \left\{(r+1)\int_{P(E)/M}\p_t\left(\frac{V}{G}\right)v_t\Xi^{r-1}\right.\\
     &\left.+\frac{(r+1)(r-1)}{2}\int_{P(E)/M}v_t^2\frac{\sqrt{-1}}{2\pi}\p\b{\p}\left(\frac{V}{G}\right)\Xi^{r-2}\right.\\
     &\left.-r\int_{P(E)/M}v_t\Xi^{r-1}\int_{P(E)/M}\p_t\left(\frac{V}{G}\right)\Xi^{r-1}\right.\\
     &\left.-r(r-1)\int_{P(E)/M}v_t\Xi^{r-1}\int_{P(E)/M}v_t\frac{\sqrt{-1}}{2\pi}\p\b{\p}\left(\frac{V}{G}\right)\Xi^{r-2}\right\}
     \frac{\omega^n}{n!}dt.
      \end{split}
   \end{align}
We will compute (\ref{1.9}) term by term by using the Stokes formula with respect to operators $\p_t$ or $\p$ and $\b\p$:
 \begin{align*}
   \begin{split}
     &\int_0^1\int_M(r+1)\left[\int_{P(E)/M}\p_t\left(\frac{V}{G}\right)v_t\Xi^{r-1}\right]\frac{\omega^n}{n!}dt\\
     &=-(r+1)\int_0^1\int_M \left[\int_{P(E)/M}\frac{V}{G}(\p_tv_t)\Xi^{r-1}
    +\frac{V}{G}v_t\p_t\Xi^{r-1}\right]\frac{\omega^n}{n!}dt
   \end{split}
 \end{align*}
 where the equality comes from $V(0)=V(1)=0$;
 \begin{align*}
   \begin{split}
     &\frac{(r+1)(r-1)}{2}\int_0^1\int_M\left[\int_{P(E)/M}v_t^2\frac{\sqrt{-1}}{2\pi}\p\b{\p}\left(\frac{V}{G}\right)
     \Xi^{r-2}\right]\frac{\omega^n}{n!}dt\\
     &=\frac{(r+1)(r-1)}{2}\int_0^1\int_M \left[\int_{P(E)/M}\frac{\sqrt{-1}}{2\pi}(\p\b{\p}v_t^2)\frac{V}{G}\Xi^{r-2}\right]\frac{\omega^n}{n!}dt\\
     &=(r+1)(r-1)\int_0^1\int_M \left[\int_{P(E)/M}\left(\frac{\sqrt{-1}}{2\pi}\p v_t\wedge\b{\p}v_t+\frac{\sqrt{-1}}{2\pi}v_t\p\b{\p}v_t\right)\frac{V}{G}\Xi^{r-2}\right]\frac{\omega^n}{n!}dt\\
     &=(r+1)\int_0^1\int_M \left[\int_{P(E)/M}\frac{V}{G}\left(\left\|\p^{\mathcal{V}}v_t\right\|^2
     \Xi^{r-1}+v_t\p_t\Xi^{r-1}\right)\right]\frac{\omega^n}{n!}dt,
   \end{split}
 \end{align*}
 where the last equality comes from (\ref{1.aaa});
 \begin{align*}
   \begin{split}
     &-r\int_0^1\int_M \left[\int_{P(E)/M}v_t\Xi^{r-1}\int_{P(E)/M}\p_t\left(\frac{V}{G}\right)\Xi^{r-1}\right]\frac{\omega^n}{n!}dt \\
     &=r\int_0^1\int_M \left[\int_{P(E)/M}v_t\Xi^{r-1}\int_{P(E)/M}\frac{V}{G}\p_t\Xi^{r-1}\right]\frac{\omega^n}{n!}dt\\
     &+r\int_0^1\int_M \left[\int_{P(E)/M}(\p_tv_t)\Xi^{r-1}\int_{P(E)/M}\frac{V}{G}\Xi^{r-1}\right]\frac{\omega^n}{n!}dt\\
     &+r\int_0^1\int_M \left[\int_{P(E)/M}\frac{(r-1){\sqrt{-1}}}{2\pi}v_t(\p\b\p v_t)\Xi^{r-2}\int_{P(E)/M}\frac{V}{G}\Xi^{r-1}\right]\frac{\omega^n}{n!}dt\\
    &=r\int_0^1\int_M \left[\int_{P(E)/M}v_t\Xi^{r-1}\int_{P(E)/M}\frac{V}{G}\p_t\Xi^{r-1}\right]\frac{\omega^n}{n!}dt\\
     &+r\int_0^1\int_M \left[\int_{P(E)/M}(\p_tv_t)\Xi^{r-1}\int_{P(E)/M}\frac{V}{G}\Xi^{r-1}\right]\frac{\omega^n}{n!}dt\\
     &-r\int_0^1\int_M \left[\int_{P(E)/M}\left\|\p^{\mathcal{V}}v_t\right\|^2\Xi^{r-1}\int_{P(E)/M}\frac{V}{G}\Xi^{r-1}\right]\frac{\omega^n}{n!}dt, \end{split}
 \end{align*}
 where the first and the second ``=" come from $V(0)=V(1)=0$ and (\ref{1.aaa}), respectively;
\begin{align*}
  \begin{split}
    &-r(r-1)\int_0^1\int_M\left[\left(\int_{P(E)/M}v_t\Xi^{r-1}\right)
    \left(\int_{P(E)/M}v_t\frac{\sqrt{-1}}{2\pi}\p\b{\p}\left(\frac{V}{G}\right)\Xi^{r-2}\right)\right]\frac{\omega^n}{n!}dt\\
           &=-r\int_0^1\int_M \left[\int_{P(E)/M}v_t\Xi^{r-1}\int_{P(E)/M}\frac{V}{G}\p_t\Xi^{r-1}\right]\frac{\omega^n}{n!}dt.
  \end{split}
\end{align*}
Thus
\begin{align}
  \begin{split}
     \frac{d\mathcal{E}(G+sV)}{ds}|_{s=0}=-\int_0^1\left(G(\p_tv_t-\left\|\p^{\mathcal{V}}v_t\right\|^2),V\right)dt.
  \end{split}
\end{align}
\end{proof}
From Theorem \ref{t1.1}, we obtain the following Euler-Lagrange equation associated to the functional $\mathcal{E}(G)$
\begin{align}\label{1.10}
  \p_tv_t-\left\|\p^{\mathcal{V}}v_t\right\|^2=0,\quad v_t=\p_t\log G_t.
\end{align}
The solutions of the equation (\ref{1.10}) are called the geodesics in $F^+(E)$, and the equation (\ref{1.10}) is called the geodesic
equation on $F^+(E)$.

A remarkable difference from the Hermitian-Einstein case, where the corresponding geodesic equation has an explicitly smooth solution for any two given end points, in our case the equation may not have any smooth solution for any two given Finsler metrics in $F^+(E)$. Fortunately, for our purpose of this paper, we only need the so-called ``approximate solutions" of (\ref{1.10}). More precisely, let $h_t$, $0\leq t\leq 1$, denote the Hermitian metric on the holomorphic vertical tangent bundle $\tilde{\mathcal{V}}$ of $P(E)$ induced by $G_t$ or $\omega_t:=2\pi\omega_{FS}(G_t)$. What we need is the existence and smoothness of the solutions of the following family Monge-Amp\`{e}re equation on $P(E)\to M$:
\begin{align}\label{ma}
(\p_tv_t-\left\|\p^{\mathcal{V}}v_t\right\|^2)\det h_t=\epsilon\det h_0,\quad v_t=\p_t\log G_t.
\end{align}
To this point, our problem is exactly the family case of the corresponding one in \cite{Chen}, see also  \cite{Semmes}. By using the geodesic approximation lemma, Lemma 7, in \cite{Chen} pointwisely for $z\in M$, and the smoothness of the solutions on the parameters $z\in M$, we have the following geodesic approximation lemma associated to the energy functional $\mathcal{E}(G)$ on $\Omega(G_0,G_1)$:
\begin{lemma}\label{ggg}(geodesic approximation lemma) The equation (\ref{ma}) has a smooth solution $G_{t,\epsilon}$ in $F^+(E)$ for any small $\epsilon>0$ and any two given endpoints $G_0, G_1\in F^+(E)$. Moreover, $G_{t,\epsilon}$ converges uniformly to a $C^{1,1}$ solution $G_t$ of the equation (\ref{1.10}) as $\epsilon\to 0$, and the following inequality holds
\begin{align}
  \left|\log \frac{G_{t,\epsilon}}{G_0}\right|\leq C,
\end{align}
where $C$ is a constant independent of $t$ and $\epsilon$.
\end{lemma}

\section{A Donaldson type functional on $F^+(E)$}\label{Donldson}

In this section, we introduce a Donaldson type functional $\mathcal{L}$ on $F^+(E)$. Originally, the Donaldson functional is defined on $Herm^+(E)$, the space of all Hermitian metrics on $E$. Fixing an arbitrary metric $H\in Herm^+(E)$, and any smooth curve $G_t$ ($0\leq t\leq 1$) in $Herm^+(E)$ with $G_0=H$ and $G_1=G$, the original Donaldson functional $\mathcal{M}(G,H)$ is
defined by (cf.  \cite{Don1}, \cite{Ko3}, \cite{Siu})
  \begin{align}\label{555}
    \mathcal{M}(G,H):=\int_M\left(\int_0^1 tr(\sqrt{-1}V_t\cdot R_t)dt-\frac{\lambda}{n}\log\det\frac{G}{H}\omega\right)\frac{\omega^{n-1}}{(n-1)!},
  \end{align}
where $V_t=(\p_tG)G^{-1}$, $R_t$ is the (Chern-) curvature of the metrics $G_t$ and the constant $\lambda$ is given by (\ref{2.000}). The definition of the Donaldson functional $\mathcal{M}(G,H)$ is independent of the choice of smooth paths in $Herm^+(E)$ connecting $H$ and $G$. It is known that the critical points of the Donaldson functional coincide with the Hermitian-Einstein metrics on $E$ at which the functional $\mathcal{M}(G,H)$ attains the absolute minimum. Especially, when $\mathcal{M}(G,H)$ is bounded below, then $E$ admits an approximate Hermitian-Einstein metric, and so $E$ is semi-stable (cf. \cite{Ko3}).

When restricted to $Herm^+(E)$, the Donaldson type functional $\mathcal{L}$ turns out to be exactly $\mathcal{M}$. By computing the first and second variations of the functional $\mathcal{L}$, we will prove that the Finsler-Einstein metrics on $E$ coincide with the critical points of this functional, and at which the Donaldson type functional $\mathcal{L}$ attains the absolute minimum. Finally by showing that the functional has a lower-bound, we obtain the main result, Theorem 0.1, in this paper.

For any fixed $H\in F^+(E)$ and any $G\in F^+(E)$, let $G(t)$, $0\leq t\leq 1$, denote a smooth curve in $F^+(E)$ with $G(0)=H$, $G(1)=G$. For example, $G(t)=(1-t)H+tG$. We still denote $v_t=\p_{t}\log G(t)$, $\Xi_t={{\sqrt{-1}}/{2\pi}}\partial\bar\partial\log G(t)$. In the following, we always omit the $t$ from $\Xi_t$ as well as the Kobayashi curvature $\Psi_t$ of $G(t)$ for simplicity.
Set
 \begin{align}\label{2.1}
   Q_1(G,H)=r\int_0^1\left[\int_{P(E)/M}v_t\Xi^{r-1}\right]dt;
 \end{align}
 \begin{align}
   Q_2(G,H)=r(r+1)\int_0^1\int_{P(E)/M}(v_t\Psi)\Xi^{r-1}dt;
 \end{align}
 \begin{align}\label{2.2}
 Q_3(G,H)=-r^2\int_{0}^{1}\left[\int_{P(E)/M}v_t\Xi^{r-1}\int_{P(E)/M}\Psi\Xi^{r-1}\right]dt.
 \end{align}
 \begin{defn} The Donaldson type functional $\mathcal{L}$ associated to a holomorphic vector bundle $E$ is defined by
 \begin{align}\label{d2.1}
   \mathcal{L}(G,H)=\int_{M}\left(Q_2(G,H)+Q_3(G,H)-\frac{\lambda}{n}Q_1(G,H)\omega\right)\frac{\omega^{n-1}}{(n-1)!}.
 \end{align}
 \end{defn}
We need to show that the Donaldson type functional $\mathcal{L}$ is well-defined, that is, $\mathcal{L}(G,H)$ is independent of the choice of
paths in $F^+(E)$ connecting $G$ with the fixed metric $H$. Actually this comes directly from the following lemma, an analogue of Lemma 6.3.6 in \cite{Ko3}, which can be proved also in a similar way.

 \begin{lemma}\label{2.22}
   Let $G_t$, $\alpha\leq t\leq \beta$, be a piecewise differentiable closed curve in $F^+(E)$, (hence $G_{\alpha}=G_{\beta}$). Set $v_t=\p_t\log G_t$, then
   \begin{align}\label{2.3}
   \int_\alpha^\beta\int_{P(E)/M}v_t\Xi^{r-1}dt=0,
   \end{align}
   \begin{align}\label{2.4}
   \int_{\alpha}^{\beta}\int_{P(E)/M}(v_t\Psi)\Xi^{r-1}dt\in \p A^{0,1}+\b{\p}A^{1,0},
   \end{align}
   \begin{align}\label{2.5}
     \int_{\alpha}^{\beta}\left(\int_{P(E)/M}v_t\Xi^{r-1}\int_{P(E)/M}\Psi\Xi^{r-1}\right)dt\in \p A^{0,1}+\b{\p}A^{1,0}.
   \end{align}
 \end{lemma}
 \begin{proof}
   The formula (\ref{2.3}) is obvious. We only prove (\ref{2.4}) while the proof of (\ref{2.5}) is similar. Let $\alpha=\alpha_0<a_1<\cdots<a_k=\beta$ be the value of $t$ where $G_t$ is not differentiable. Fix a reference point $H\in F^+(E)$. It suffices then to prove (\ref{2.4}) for the closed curve consisting of a smooth curve from $H$ to $G_{a_j}$, the curve $G_t$, $a_j\leq t\leq a_{j+1}$, and a smooth curve from $G_{a_{j+1}}$ back to $H$. Set $a=a_{j}$, $b=a_{j+1}$,
   $$\t=\{(t,s):a\leq t\leq b,\ 0\leq s\leq 1\}.$$
   Let $G:\t\to F^+(E)$ be a smooth map such that
   $$G(t,0)=H,\quad G(t,1)=G_t,\quad\mbox{\rm for}\ a\leq t\leq b.$$
   Denote by
   $$\phi=r\int_{P(E)/M}(\tilde{d}\log G\wedge\Psi)\w\Xi^{r-1},$$
   where $\tilde{d}=(\p/\p s)ds+(\p/\p t)dt$.  Note that
   \begin{align}\label{2.6}
   \begin{split}
     \int_{\p\t}\phi&=\int_{t=a}^{t=b}\phi|_{s=0}+\int_{s=0}^{s=1}\phi|_{t=b}-\int_{t=a}^{t=b}\phi|_{s=1}-\int_{s=0}^{s=1}\phi|_{t=a}\\
     &=-\int_{a}^{b}(\int_{P(E)/M} 2\pi v_t\Xi^r)dt+\frac{1}{r+1}(Q_2(G_b,H)-Q_2(G_a,H)),
     \end{split}
   \end{align}
  where $\p\t$ is oriented counterclockwise. On the other hand,
   \begin{align*}
     \frac{1}{r}\tilde{d}\phi&=-\frac{1}{2}\left[\p \int_{P(E)/M}(\tilde{d}\log G\wedge\b{\p}\tilde{d}\log G)\wedge\Xi^{r-1}\right.\\
     &\left.-\b{\p}\int_{P(E)/M}(\p\tilde{d}\log G\wedge{\tilde d}\log G)\wedge\Xi^{r-1}\right]\in \p A^{0,1}+\b{\p}A^{1,0},
   \end{align*}
  so  statement (\ref{2.4}) follows from the following Stokes formula:
   \begin{align}\label{1.1222}
   \int_{\p\t}\phi=\int_{\t}\tilde{d}\phi.
   \end{align}
   \end{proof}
Actually, by the following proposition, one can see that the functional $\mathcal{L}$ is indeed an extension of the Donaldson functional $\mathcal{M}$.
\begin{prop}\label{pp} For any fixed $H\in Herm^+(E)$ and any $G\in Herm^+(E)$,
\begin{align}\label{999}
\mathcal{L}(G,H)=\mathcal{M}(G,H).
\end{align}
\end{prop}
\begin{proof} Let $G(t)=(1-t)H+t G(\in Herm^+(E))$. Then
 \begin{align*}
   Q_1(G,H)&=r\int_0^1\int_{P(E)/M}v_t\Xi^{r-1}dt=r\int_0^1\int_{P(E)/M}(\p_t\log G)\Xi^{r-1}dt\\
   &=r\int_0^1\p_tG_{i\b{j}}\int_{P(E)/M}\frac{v^i \b{v}^j}{G}\Xi^{r-1}dt=\int_0^1\p_t\log\det(G_{i\b j}(t))dt=\log\det\frac{G}{H};
 \end{align*}
 \begin{align*}
   Q_2(G,H)&=(r+1)r\int_0^1\int_{P(E)/M}v_t\Psi\Xi^{r-1}dt\\
   &=(r+1)r\int_0^1\p_t G_{i\b{j}}(t)\int_{P(E)/M}\frac{v^i \b{v}^j}{G(t)}\Psi\Xi^{r-1}dt\\
   &=\sqrt{-1}(r+1)r\int_0^1\p_t G_{i\b{j}}(t)K_{k\b{l}\alpha\b{\beta}}dz^{\alpha}\wedge d\b{z}^{\beta}\frac{G^{i\b{j}}G^{k\b{l}}+G^{i\b{l}}G^{k\b{j}}}{r(r+1)}dt\\
   &=\sqrt{-1}\int_0^1(\p_t\log\det G)(\b{\p}\p\log\det G)+tr(V_t\cdot R_t))dt,
 \end{align*}
 \begin{align*}
   Q_3(G,H)&=-r^2\int_0^1\left[\int_{P(E)/M}v_t\Xi^{r-1}\int_{P(E)/M}\Psi\Xi^{r-1}\right]dt\\
   &=-r\int_0^1\left[(\p_t\log\det G){\sqrt{-1}}K_{i{\b j}\alpha\b\beta}dz^\alpha\w d\b z^\beta\int_{P(E)/M}\frac{v^i \b{v}^j}{G}\Xi^{r-1}\right]dt\\
   &=-\sqrt{-1}\int_0^1(\p_t\log\det G)(\b{\p}\p\log\det G)dt.
 \end{align*}
 Thus
 \begin{align*}
   \mathcal{L}(G,H)&=\int_{M}(Q_{2}(G,H)+Q_3(G,H)-\frac{\lambda}{n}Q_{1}(G,H)\w\omega)\wedge\frac{\omega^{n-1}}{(n-1)!}\\
   &=\int_M(\sqrt{-1}\int_0^1 tr(V_t\cdot R_t)dt-\frac{\lambda}{n}(\log\det\frac{G}{H})\omega)\wedge\frac{\omega^{n-1}}{(n-1)!}=\mathcal{M}(G,H).
 \end{align*}
\end{proof}
As the Donaldson functional on $Herm^+(E)$, the Donaldson type functional $\mathcal{L}$ on $F^+(E)$ also has the following property analogous
to Lemma 6.3.23 in \cite{Ko3}.
\begin{lemma} Let $G,G',G''\in F^+(E)$. Then
 \begin{align}\label{2.8}
 \mathcal{L}(G,G')+\mathcal{L}(G',G'')+\mathcal{L}(G'',G)=0;
 \end{align}
 \begin{align}\label{2.88}
 \mathcal{L}(G,aG)=0\ \mbox{\rm for any positive constant}\ a;
 \end{align}
 \begin{align}\label{2.89}
 \begin{split}
     &\sqrt{-1}\p\b{\p}(Q_{2}(G,G')+Q_3(G,G')-\frac{\lambda}{n}Q_1(G,G')\omega)\\
     &=-\frac{1}{(2\pi)^{r-1}}(s_2(E,G)-s_2(E,G'))+\frac{1}{(2\pi)^{2r-2}2}(s_1(E,G)^2-s_1(E,G')^2)\\
     &-\frac{\lambda}{(2\pi)^{r-1}n}(s_1(E,G)-s_1(E,G'))\wedge\omega,
 \end{split}
 \end{align}
 where $s_j(E,G)=\int_{P(E)/M}(2\pi\Xi)^{r-1+j}$, $1\leq j\leq n$, is the $j$-th Segre form (cf. \cite{FLW}).
\end{lemma}
\begin{proof} We first give the proof of (\ref{2.8}). By Lemma \ref{2.2}, we have
   $$Q_{i}(G,G')+Q_{i}(G',G'')+Q_i(G'',G)\equiv 0\quad \mod \quad \p A^{0,1}+\b{\p}A^{1,0}, \quad i=2,3,$$
   $$Q_{1}(G,G')+Q_{1}(G',G'')+Q_1(G'',G)=0.$$
 For the proof of (\ref{2.88}), we let $G_t=e^{(1-t)b}G$, $a=e^b$. So we have
   $$Q_1(G,aG)=-br\int_0^1 \int_{P(E)/M}\Xi^{r-1}dt=-br,$$
   $$Q_2(G,aG)=2\pi b(r+1)\int_0^1 \int_{P(E)/M}\Xi^{r}dt,$$
   $$Q_3(G,aG)=-2\pi br\int_0^1 \int_{P(E)/M}\Xi^{r}dt.$$
   Therefore
   $$[Q_2(G,aG)]=-2\pi b(r+1)c_1(E),\quad [Q_3(G,aG)]=2\pi br c_1(E),$$
   and  $\mathcal{L}(G,aG)=0$.

   For the proof of (\ref{2.89}), we note that, since
   $$\sqrt{-1}\p\b{\p}Q_2(G,G')=-\frac{1}{(2\pi)^{r-1}}\int_{P(E)/M}((\sqrt{-1}\p\b{\p}\log G)^{r+1}-(\sqrt{-1}\p\b{\p}\log G')^{r+1}),$$
   $$\sqrt{-1}\p\b{\p}Q_1(G,G')=\frac{1}{(2\pi)^{r-1}}\int_{P(E)/M}((\sqrt{-1}\p\b{\p}\log G)^{r}-(\sqrt{-1}\p\b{\p}\log G')^{r}),$$
   and
   \begin{align*}
     \sqrt{-1}\p\b{\p}Q_3(G,G')=\frac{1}{2(2\pi)^{2r-2}}[(\int_{P(E)/M}(\sqrt{-1}\p\b{\p}\log G)^r)^2-(\int_{P(E)/M}(\sqrt{-1}\p\b{\p}\log G')^r)^2],
   \end{align*}
   we have
   \begin{align*}
     &\sqrt{-1}\p\b{\p}(Q_{2}(G,G')+Q_3(G,G')-\frac{\lambda}{n}Q_1(G,G')\omega)=-\frac{1}{(2\pi)^{r-1}}(s_2(E,G)-s_2(E,G'))\\
     &+\frac{1}{(2\pi)^{2r-2}2}(s_1(E,G)^2-s_1(E,G')^2)-
     \frac{\lambda}{(2\pi)^{r-1}n}(s_1(E,G)-s_1(E,G'))\wedge\omega.
   \end{align*}
   The lemma is proved.
 \end{proof}

Now we compute the first variation of the Donaldson type functional $\mathcal{L}(G_t,H)$. Let $G_t, a\leq t\leq b$, be any differentiable curve in $F^+(E)$ and $H$ any fixed point of $F^+(E)$. From the proof of Lemma \ref{2.2}, we have
   $$r\int_{a}^{b}(\int_{P(E)/M}v_t\Xi^{r-1})dt+Q_1(G_a,H)-Q_1(G_b,H)\in \p A^{0,1}+\b{\p}A^{1,0},$$
   $$(r+1)r\int_{a}^{b}(\int_{P(E)/M}v_t\Psi\Xi^{r-1})dt+Q_2(G_a,H)-Q_2(G_b,H)\in \p A^{0,1}+\b{\p}A^{1,0},$$
  $$-r^2\int_{a}^{b}\left(\int_{P(E)/M}v_t\Xi^{r-1}\int_{P(E)/M}\Psi\Xi^{r-1}\right)dt+Q_3(G_a,H)-Q_3(G_b,H)
     \in \p A^{0,1}+\b{\p}A^{1,0}.$$
Differentiating the above formulas with respect to $b$, we obtain the following formulas:
\begin{align}\label{3.1}
     \p_t Q_1(G_t,H)= r\int_{P(E)/M}v_t\Xi^{r-1};
   \end{align}
   \begin{align}\label{3.2}
   \p_t Q_2(G_t,H)\equiv (r+1)r\int_{P(E)/M}v_t\Psi\Xi^{r-1}\quad \mod \quad \p A^{0,1}+\b{\p}A^{1,0};
   \end{align}
   \begin{align}\label{3.3}
   \p_t Q_3(G_t,H)&\equiv -r^2\int_{P(E)/M}v_t\Xi^{r-1}\int_{P(E)/M}\Psi\Xi^{r-1}\quad \mod \quad \p A^{0,1}+\b{\p}A^{1,0}.
 \end{align}
Therefore, one has
 \begin{align*}
   \frac{1}{r}\frac{d \mathcal{L}(G_t,H)}{d t}&=\int_M \left[(r+1)\int_{P(E)/M}v_t\Psi\Xi^{r-1}-r\int_{P(E)/M}v_t\Xi^{r-1}\int_{P(E)/M}\Psi\Xi^{r-1}\right.\\
   &\left.-\frac{\lambda}{n}\int_{P(E)/M}v_t\Xi^{r-1}\omega\right]\wedge\frac{\omega^{n-1}}{(n-1)!}\\
   &=\int_M \left[(r+1)\int_{P(E)/M}v_t tr_{\omega}\Psi\Xi^{r-1}-r\int_{P(E)/M}v_t\Xi^{r-1}\int_{P(E)/M}tr_{\omega}\Psi\Xi^{r-1}\right.\\
   &\left.-\lambda\int_{P(E)/M}v_t\Xi^{r-1}\right]\frac{\omega^{n}}{n!}\\
   &=\int_{M} \left[(r+1)\int_{P(E)/M}v_t (tr_{\omega}\Psi-\lambda)\Xi^{r-1}\right.\\
   &\left.-r\int_{P(E)/M}v_t\Xi^{r-1}\int_{P(E)/M}(tr_{\omega}\Psi-\lambda)\Xi^{r-1}\right]\frac{\omega^{n}}{n!},
 \end{align*}
from which we get the first variation formula of the Donaldson type functional $\mathcal{L}$ on $F^+(E)$:
 \begin{align}\label{2.8888}
 \frac{d \mathcal{L}(G_t,H)}{d t}=r(v_t G, (tr_{\omega}\Psi-\lambda)G).
 \end{align}

Now we compute the second variation  of $L(G_t,H)$. By a direct computation, we have
\begin{align}\label{8.1}
  \begin{split}
    &\frac{1}{r}\frac{d^2 \mathcal{L}(G_t,H)}{d t^2}=\int_M\left[(r+1)\int_{P(E)/M}\p_tv_t(tr_{\omega}\Psi-\lambda)\Xi^{r-1}\right.\\
    &+(r+1)\int_{P(E)/M}v_t\p_t((tr_{\omega}\Psi-\lambda)\Xi^{r-1})-r\int_{P(E)/M}\p_tv_t\Xi^{r-1}\int_{P(E)/M}(tr_{\omega}\Psi-\lambda)\Xi^{r-1}\\
    &\left.-r\int_{P(E)/M}v_t\p_t\Xi^{r-1}\int_{P(E)/M}(tr_{\omega}\Psi-\lambda)\Xi^{r-1}
    -r\int_{P(E)/M}v_t\Xi^{r-1}\int_{P(E)/M}\p_t\left((tr_{\omega}\Psi-\lambda)\Xi^{r-1}\right)\right]\frac{\omega^n}{n!}.
  \end{split}
\end{align}
We compute the terms in (\ref{8.1}) one by one:
\begin{align}\label{8.2}
  \begin{split}
    &\int_{M}\left[\int_{P(E)/M}v_t\p_t\left((tr_{\omega}\Psi-\lambda)\Xi^{r-1}\right)\right]\frac{\omega^n}{n!}
    =\int_{M}\left[\int_{P(E)/M}-\frac{2\pi}{r}tr_{\omega}v_t\p_t\Xi^r-\lambda\int_{P(E)/M}v_t\p_t\Xi^{r-1}\right]\frac{\omega^n}{n!}\\
    &=\int_{M}\left[tr_{\omega}\int_{P(E)/M}\sqrt{-1}\p v_t\wedge\b{\p}v_t\Xi^{r-1}-\lambda\int_{P(E)/M}v_t\p_t\Xi^{r-1}\right]\frac{\omega^n}{n!}\\
    &=\int_{P(E)}\left(\left\|\p^{\mathcal{H}}v_t\right\|^2-(tr_{\omega}\Psi-\lambda)\left\|\p^{\mathcal{V}}v_t\right\|^2\right)\Xi^{r-1}\wedge\frac{\omega^{n}}{n!},
  \end{split}
\end{align}
where the last equality comes from (\ref{1.aaa}) and (\ref{1.bbb});
\begin{align}\label{8.3}
  \begin{split}
    &\int_{M}\left[\int_{P(E)/M}v_t\Xi^{r-1}\int_{P(E)/M}\p_t\left((tr_{\omega}\Psi-\lambda)\Xi^{r-1}\right)\right]\frac{\omega^n}{n!}\\
    &=\int_{M}tr_\omega\left[\int_{P(E)/M}\p v_t\Xi^{r-1}\int_{P(E)/M}\b{\p}v_t\Xi^{r-1}\right]\frac{\omega^n}{n!}\\
    &=\int_{M}g^{\alpha\b{\beta}}\left(\int_{P(E)/M}\frac{\delta v_t}{\delta z^{\alpha}}\Xi^{r-1}\int_{P(E)/M}\frac{\delta v_t}{\delta \b{z}^{\beta}}\Xi^{r-1}\right)\frac{\omega^n}{n!}\\
&=\int_{M}\left\|\int_{P(E)/M}\p^{\mc{H}}v_t\Xi^{r-1}\right\|^2\frac{\omega^n}{n!},
  \end{split}
\end{align}
where the last equality comes from (\ref{1.22233});
\begin{align}\label{8.4}
  \begin{split}
    &\int_{M}\left[\int_{P(E)/M}v_t\p_t\Xi^{r-1}\int_{P(E)/M}(tr_{\omega}\Psi-\lambda)\Xi^{r-1}\right]\frac{\omega^n}{n!}\\
    &=(r-1)\int_{M}\left[\int_{P(E)/M}v_t\frac{\sqrt{-1}}{2\pi}\p\b{\p}v_t\Xi^{r-2}\int_{P(E)/M}
    (tr_{\omega}\Psi-\lambda)\Xi^{r-1}\right]\frac{\omega^n}{n!}\\
    &=-\int_{M}\left[\int_{P(E)/M}\left\|\p^\mathcal{V}v_t\right\|^2\Xi^{r-1}\int_{P(E)/M}(tr_{\omega}\Psi-\lambda)\Xi^{r-1}\right]\frac{\omega^n}{n!}.
  \end{split}
\end{align}
So we have from (\ref{8.1})-(\ref{8.4}):
\begin{align*}
  \begin{split}
    &\frac{1}{r}\frac{d^2 \mathcal{L}(G_t,H)}{d t^2}=\int_M\left[(r+1)\int_{P(E)/M}\p_tv_t(tr_{\omega}\Psi-\lambda)\Xi^{r-1}
    +(r+1)\int_{P(E)/M}\left\|\p^{\mc{H}}v_t\right\|^2\Xi^{r-1}\right.\\
    &\left.-(r+1)\int_{P(E)/M}\left\|\p^\mathcal{V}v_t\right\|^2(tr_{\omega}\Psi-\lambda)\Xi^{r-1}-r\int_{P(E)/M}\p_t v_t\Xi^{r-1}\int_{P(E)/M}(tr_{\omega}\Psi-\lambda)\Xi^{r-1}\right.\\
    &\left.+r\int_{P(E)/M}\left\|\p^\mathcal{V}v_t\right\|^2\Xi^{r-1}\int_{P(E)/M}(tr_{\omega}\Psi-\lambda)\Xi^{r-1}
    -r\left\|\int_{P(E)/M}\p^{\mc{H}}v_t\Xi^{r-1}\right\|^2\right]\frac{\omega^n}{n!}\\
    &=\left((\p_tv_t-\left\|\p^\mathcal{V}v_t\right\|^2)G, (tr_{\omega}\Psi-\lambda)G\right)
    +\int_{M}|(\p^\mathcal{H}v_t)G|^2_{\omega,G}\frac{\omega^n}{n!}\\
    &=\left((\p_tv_t-\left\|\p^\mathcal{V}v_t\right\|^2)G, (tr_{\omega}\Psi-\lambda)G\right)+\left\|(\p^{\mathcal{H}}v_t)G\right\|_G^2.
  \end{split}
\end{align*}
Here,
\begin{align}
  (\p^{\mathcal{H}}v_t)G=\frac{\delta v_t}{\delta z^{\alpha}}dz^{\alpha}\otimes G
\end{align}
 can be viewed as an element of $\tilde{\mathcal{H}}^*\otimes F^+E$, ($\tilde{\mathcal{H}}^*$ denotes the dual bundle of the horizontal tangent bundle $\tilde{\mathcal{H}}$), which has a natural fiberwise metric induced from $\omega$ and $\langle\cdot,\cdot\rangle_G$, i.e.,
\begin{align}\label{1.111}
\begin{split}
  |(\p^\mathcal{H}v_t)G|^2_{\omega,G}&:=(r+1)\int_{P(E)/M}g^{\alpha\b{\beta}}\frac{\delta v_t}{\delta z^{\alpha}}\frac{\delta v_t}{\delta \b{z}^{\beta}}\Xi^{r-1}\\
  &\quad-r g^{\alpha\b{\beta}}\left(\frac{\p}{\p z^{\alpha}}\int_{P(E)/M} v_t\Xi^{r-1}\right)
  \left(\frac{\p}{\p \b{z}^{\beta}}\int_{P(E)/M} v_t\Xi^{r-1}\right)\\
  &=(r+1)\int_{P(E)/M}g^{\alpha\b{\beta}}\frac{\delta v_t}{\delta z^{\alpha}}\frac{\delta v_t}{\delta \b{z}^{\beta}}\Xi^{r-1}\\
  &\quad-r g^{\alpha\b{\beta}}\left(\int_{P(E)/M} \frac{\delta v_t}{\delta z^{\alpha}}\Xi^{r-1}\right)
  \left(\int_{P(E)/M} \frac{\delta v_t}{\delta \b{z}^{\beta}}\Xi^{r-1}\right)\geq 0,
\end{split}
\end{align}
where the second equality holds since
\begin{align}
\begin{split}
  &\quad g^{\alpha\b{\beta}}\left(\frac{\p}{\p z^{\alpha}}\int_{P(E)/M} v_t\Xi^{r-1}\right)
  \left(\frac{\p}{\p \b{z}^{\beta}}\int_{P(E)/M} v_t\Xi^{r-1}\right)\\
  &=tr_{\omega}\left(\p\int_{P(E)/M}v_t\Xi^{r-1}\right)\left(\b{\p}\int_{P(E)/M}v_t\Xi^{r-1}\right)\\
  &=tr_{\omega}\int_{P(E)/M}\p v_t\Xi^{r-1}\int_{P(E)/M}\b{\p} v_t \Xi^{r-1}\\
  &=g^{\alpha\b{\beta}}\left(\int_{P(E)/M} \frac{\delta v_t}{\delta z^{\alpha}}\Xi^{r-1}\right)
  \left(\int_{P(E)/M} \frac{\delta v_t}{\delta \b{z}^{\beta}}\Xi^{r-1}\right).
\end{split}
\end{align}

That is, the second variation formula of the Donaldson type functional $\mathcal{L}$ is given by
\begin{align}\label{8.8}
\frac{d^2 \mathcal{L}(G_t,H)}{d t^2}=r\left((\p_tv_t-\left\|\p^\mathcal{V}v_t\right\|^2)G, (tr_{\omega}\Psi-\lambda)G\right)+r\left\|(\p^{\mathcal{H}}v_t)G\right\|_G^2.
\end{align}

Using the first and second variation formulas (\ref{2.8888}) and (\ref{8.8}), we obtain the following
\begin{thm}\label{p.1} For fixed $H\in F^+(E)$, the metric $G\in F^+(E)$ is a critical point of the Donadlson type functional $\mathcal{L}(G)=\mathcal{L}(G,H)$ on $F^+(E)$ if and only if $tr_{\omega}\Psi-\lambda=0$, i.e., $G$ is a Finsler-Einstein structure on $E$.
Moreover, the Donadlson type functional $\mathcal{L}(G,H)$ attains a local minimum
at each Finsler-Einstein metric $G$ on $E$.
\end{thm}
\begin{proof} Clearly, if $G$ is a Finsler-Einstein structure on $E$, i.e. $tr_{\omega}\Psi-\lambda=0$, then it is a critical point of $\mathcal{L}(G,H)$ by the equation (\ref{2.8888}). Moreover, at a Finsler-Einstein metric $G$ on $E$, by (\ref{8.8}), one has
\begin{align}\label{5.55}
\frac{d^2\mathcal{L}(G_t,H)}{d t^2}|_{t=0}=r\left\|(\p^{\mathcal{H}}v_t)G\right\|_G^2\geq 0,
\end{align}
that is, the functional $\mathcal{L}(G,H)$ attains a local minimum. Conversely, let $G$ be a critical point of $\mathcal{L}(G,H)$. By choosing the variation
 \begin{align}\label{2.1111}
   v_t=\p_t\log G=-(tr_{\omega}\Psi-\lambda),
 \end{align}
we get
$$0=\frac{d\mathcal{L}(G_t,H)}{d t}=-r((tr_{\omega}\Psi-\lambda)G,(tr_{\omega}\Psi-\lambda)G)\leq 0,$$
and from which we have
$$tr_{\omega}\Psi-\lambda=0,$$
that is, the metric $G$ is a Finsler-Einstein metric.
 \end{proof}
In fact, the Donaldson type functional $\mathcal{L}(G,H)$ attains the absolute minimum at all Finsler-Einstein metrics on $E$.
\begin{thm}\label{p.2} For a fixed Finsler metric $H\in F^+(E)$, the Donaldson type functional $\mathcal{L}(G,H)$ on the space $F^+(E)$ attains the absolute minimum at all Finsler-Einstein metrics on $E$.
\end{thm}
\begin{proof} For any Finsler-Einstein metric $G_0$ on $E$ and any metric $G_1\in F^+(E)$, we have by Lemma \ref{ggg} that for any small $\epsilon>0$, there exists a smooth curve $G_{t,\epsilon}$ in $F^+(E)$ joining $G_0$ and $G_1$ such that
 $$(\p_tv_{t,\epsilon}-\left\|\p^{\mathcal{V}}v_{t,\epsilon}\right\|^2)\det h_{t,\epsilon}=\epsilon\det h_0,\ v_{t,\epsilon}=\p_t\log G_{t,\epsilon}.$$
Moreover, $G_{t,\epsilon}$ converges uniformly to a $C^{1,1}$ solution $G_t$ of the equation (\ref{1.10}).
On the other hand, from (\ref{8.8}), we have

\begin{align}
  \begin{split}
    &\frac{1}{r}\frac{d^2\mc{L}(G_{t,\epsilon},H)}{dt^2}
    \geq \left((\p_tv_{t,\epsilon}-\left\|\p^{\mathcal{V}}v_{t,\epsilon}\right\|^2)G_{t,\epsilon}, (tr_{\omega}\Psi_{t,\epsilon}-\lambda)G_{t,\epsilon}\right)\\
    &=\int_M\left[(r+1)\int_{P(E)/M}\left(\p_tv_{t,\epsilon}-\left\|\p^{\mathcal{V}}v_{t,\epsilon}\right\|^2\right)
    (tr_{\omega}\Psi_{t,\epsilon}-\lambda)\omega^{r-1}_{FS}
    \right.\\
    &\left.-r\int_{P(E)/M}\left(\p_tv_{t,\epsilon}-\left\|\p^{\mathcal{V}}v_{t,\epsilon}\right\|^2\right)\omega^{r-1}_{FS}
    \int_{P(E)/M}(tr_{\omega}\Psi_{t,\epsilon}-\lambda)\omega^{r-1}_{FS}\right]\frac{\omega^n}{n!}\\
    &=\int_M\left[(r+1)\int_{P(E)/M}(tr_{\omega}\Psi_{t,\epsilon}-\lambda)\epsilon\omega^{r-1}_{FS}(h_0)
  -r\int_{P(E)/M}\epsilon\omega^{r-1}_{FS}(h_0)\int_{P(E)/M}(tr_{\omega}\Psi_{t,\epsilon}-\lambda)\omega^{r-1}_{FS}\right]\frac{\omega^n}{n!}\\
  &=(r+1)\epsilon\int_M\int_{P(E)/M}tr_{\omega}\Psi_{t,\epsilon}\omega^{r-1}_{FS}(h_0)\frac{\omega^n}{n!}-(r+1)\epsilon\lambda\text{Vol}(M),
  \end{split}
\end{align}
where ${\rm Vol}(M)=\int_M\frac{\omega^n}{n!}$.
In the following we show that the integral $\int_M\int_{P(E)/M}tr_{\omega}\Psi_{t,\epsilon}\omega^{r-1}_{FS}(h_0)\frac{\omega^n}{n!}$ is bounded below.
For this, we take a small open set $U$ in $P(E)$ with the homogenous coordinate $(z^1,\cdots,z^n,v^1,\cdots, v^r)$. For any smooth function $\rho$ on $P(E)$ with  compact support in $U$, we have for some $C_\rho>0$ that
\begin{align*}
&\int_M\int_{P(E)/M}\rho tr_{\omega}\Psi_{t,\epsilon}\frac{\omega^{r-1}_{FS}(h_0)}{(r-1)!}\frac{\omega^n}{n!}
=\int_U\rho tr_{\omega}\Psi_{t,\epsilon}\frac{\omega^{r-1}_{FS}(h_0)}{(r-1)!}\frac{\omega^n}{n!}\\
  &=\int_U\rho tr_{\omega}\Psi_{t,\epsilon}\det h_0\det g  dV\\
  &=\int_U\rho g^{\alpha\b{\beta}}\left(-\p_{\alpha}\p_{\b{\beta}}\log G_{t,\epsilon}
  +\sum_{i,j=1}^{r-1}(\log G_{t,\epsilon})^{i\b{j}}(\p_{\alpha}\p_{\b{j}}\log G_{t,\epsilon})(\p_i\p_{\b\beta}\log G_{t,\epsilon})\right)\det h_0\det g dV\\
  &\geq-\int_U(\rho g^{\alpha\b{\beta}}\det h_0\det g)\p_{\alpha}\p_{\b{\beta}}\log G_{t,\epsilon} dV\\
  &=-\int_U(\rho g^{\alpha\b{\beta}}\det h_0\det g)\p_{\alpha}\p_{\b{\beta}}\log\frac{G_{t,\epsilon}}{G_0} dV-\int_U(\rho g^{\alpha\b{\beta}}\det h_0\det g)\p_{\alpha}\p_{\b{\beta}}\log G_0 dV\\
  &=-\int_U\left(\p_{\alpha}\p_{\b{\beta}}(\rho g^{\alpha\b{\beta}}\det h_0\det g)\right)\log \frac{G_{t,\epsilon}}{G_0} dV-\int_U(\rho g^{\alpha\b{\beta}}\det h_0\det g)\p_{\alpha}\p_{\b{\beta}}\log G_0 dV\\
  &\geq -C_\rho,
\end{align*}
where
$$dV=\frac{1}{\det h_0\det g}{\frac{\omega^{r-1}_{FS}(h_0)}{(r-1)!}}\frac{\omega^n}{n!}.$$
Now by a partition of unity argument for the compact manifold $P(E)$, one gets easily that
\begin{align}\label{8.10}
\int_M\int_{P(E)/M}tr_{\omega}\Psi_{t,\epsilon}\omega_{FS}^{r-1}(h_0)\frac{\omega^n}{n!}\geq -C_1
\end{align}
holds for some positive constant $C_1$. Hence for any small $\epsilon>0$, we have
\begin{align}\label{8.12}
\frac{d^2 \mc{L}(G_t,H)}{d t^2}\geq -C_2\epsilon
\end{align}
for some positive constant $C_2$. Now since $G_0$ is a Finsler-Einstein metric, we get
\begin{align}\label{8.13}
\frac{d \mc{L}(G_t,H)}{d t}=\int_0^t\frac{d^2\mc{L}(G_t,H)}{d t^2}dt\geq-\epsilon C_2t,
\end{align}
and so
\begin{align}\label{8.14}
\mc{L}(G_1,H)-\mc{L}(G_0,H)=\int_0^1\frac{d\mc{L}(G_t,H)}{d t} dt\geq-\frac{\epsilon C_2}{2}.
\end{align}
Taking $\epsilon\to 0$, we have
\begin{align}\label{8.15}
\mc{L}(G_1,H)\geq \mc{L}(G_0,H).
\end{align}
Finally, if $G_1$ is also a Finsler-Einstein metric on $E$, we get $\mathcal{L}(G_0,H)=\mathcal{L}(G_1, H)$. $\quad\quad\quad\Box$

$$  \ $$

\noindent{\bf The proof of Theorem 0.1:}  Note that for a fixed Hermitian metric $H$ and any Finsler-Einstein metric $G_0$ on $E$, we have by Proposition \ref{pp} and Theorem \ref{p.2},
   \begin{align}\label{8.16}
   \inf_{G\in Herm^+(E)}\mathcal{M}(G,H)=\inf_{G\in Herm^+(E)}\mathcal{L}(G,H)\geq \inf_{G\in F^+(E)}\mathcal{L}(G,H)= \mathcal{L}(G_0,H)>-\infty.
   \end{align}
Thus the Donaldson functional $\mathcal{M}(G,H)$ is bounded below. Now the theorem follows from Theorem 6.10.13 in \cite{Ko3}.
 \end{proof}
\begin{rem} As we point out in the introduction of this paper, a natural question is whether a semi-stable vector bundle $E$ admits a Finsler-Einstein metric. If not, then by Theorem 0.1, the condition of the existence of Finsler-Einstein metrics would be a strong condition for the semi-stability. We may ask what the algebraic-geometric counterpart of the notion of a Finsler-Einstein metric is.
\end{rem}

\end{document}